\documentclass[11pt]{amsart}
\usepackage{amsfonts,amssymb,latexsym,amsmath, amsxtra}
\usepackage[all]{xy}
\usepackage[dvips]{graphics}

\usepackage[OT2,T1]{fontenc}
\DeclareSymbolFont{cyrletters}{OT2}{wncyr}{m}{n}
\DeclareMathSymbol{\Sha}{\mathalpha}{cyrletters}{"58}

\theoremstyle{plain}
\newtheorem{theorem}{Theorem}[section]

\newtheorem{lemma}[theorem]{Lemma}

\newtheorem*{conjecture*}{Conjecture}

\theoremstyle{definition}

\theoremstyle{remark}
\newtheorem*{remark}{Remark}
\newtheorem*{remarks}{Remarks}

\numberwithin{equation}{section}
\newcommand{\Qp}{\widehat{\overline{\Q}}_p}

\newcommand{\R}{\mathbb R}
\newcommand{\N}{\mathbb N}
\newcommand{\Z}{\mathbb Z}
\newcommand{\C}{\mathbb C}

\newcommand{\Q}{{\mathbb Q}}

\def\ord{\text{ord}}

\def\H{\mathbb H}

\def\SL{\rm SL}

\def\({\left(}
\def\){\right)}

\newcommand{\im}[1]{\text{Im}\(#1\)}

\newcommand{\FF}{\mathcal{F}}

\def\k2{\frac{k}{2}}

\begin{document}

\title{Explicit congruences for mock modular forms}

\author{Ben Kane}
\address{Department of Mathematics\\ University of Hong Kong\\ Pokfulam, Hong Kong}
\email{bkane@hku.hk}

\author{Matthias Waldherr}
\address{Schwebheimerstr. 10\\
97469 Gochsheim, Germany}
\thanks{ The research of the first author was supported by grant project numbers 27300314 and 17302515 of the Research Grants Council.  Part of this research was conducted while the first author was a postdoc at the University of Cologne and the second author was a Ph.D. student at the University of Cologne}

\date{\today}
\thispagestyle{empty} \vspace{.5cm}

\begin{abstract}
In recent work of Bringmann, Guerzhoy, and the first author, $p$-adic modular forms were constructed from mock modular forms.  This paper proves explicit congruences for these $p$-adic modular forms.
\end{abstract}
\maketitle

\section{Introduction and statement of results}

Let $2\leq k\in 2\Z$ be given and let $S_k$ denote the space of weight $k$ cusp forms.  Given a newform $g\in S_k$, there is a closely linked weight $2-k$ non-holomorphic modular form $M$, known as a harmonic weak Maass form, which maps to $g/\|g\|^2$ under the differential operator $\xi_{2-k}:=2iy^{2-k} \overline{\frac{\partial}{\partial\overline{z}}}$, where $\|\cdot\|$ is the usual Petersson norm and $z=x+iy$ in the usual splitting in the upper half-plane $\H$.  The function $M$ further naturally splits into a holomorphic part $M^+$, referred to as a mock modular form, and a non-holomorphic part $M^{-}$.  Both parts have a Fourier expansion, which can be expressed as a $q$-series for the holomorphic part.  Although the coefficients of $M^+$ are generally expected to be transcendental, Guerzhoy, Kent, and Ono \cite{GKO} proved that there exists an $\alpha_1\in\C$ such that 
$$
\mathcal{F}_{\alpha_1}:=M^+-\alpha_1E_g
$$
with algebraic coefficients by subtracting a certain (likely transcendental) constant multiple of the (holomorphic) Eichler integral $E_g$ corresponding to $g$.   From this, they construct a family of formal power series $\mathcal{F}_{\alpha}$ by adding $\alpha E_g$ for algebraic $\alpha$.  They then $p$-adically link the coefficients of $\mathcal{F}_{\alpha}$ to the coefficients of $g$ except for at most one exceptional choice of $\alpha$.  Constructing this exceptional choice of $\alpha$ within an algebraic closure $\widehat{\overline{Q_p}}$ of the $p$-adic completion of $\Q_p$ and further correcting the series with another Eichler integral integral closely related to $g$, Bringmann, Guerzhoy, and the first author \cite{BGK} considered a further modification 
$$
\mathcal{F}_{\alpha,\delta}:=\mathcal{F}_{\alpha}-\delta E_{g|V_p},
$$
where $\delta\in\widehat{\overline{Q_p}}$.  This refinement turns out to be a $p$-adic modular form for a unique choice of $\alpha$ and $\delta$ whenever the $p$-order $\ord_p(\lambda_p)$ of the $p$th Hecke eigenvalue $\lambda_p$ of $g$ satisfies $0<\ord_p(\lambda_p)<(k-1)/2$.  What is worth noting here is that although the mock modular form is not itself modular (and likely has transcendental coefficients), the $p$-adic modular form $\mathcal{F}_{\alpha,\delta}$ is congruent to weakly holomorphic modular forms (those meromorphic modular forms whose only possible poles are occur at cusps).  The poles of these weakly holomorphic modular forms increase with the power of the prime $p$ in the congruence, but the weight remains $k$.

The purpose of this paper is to construct congruences between $\mathcal{F}_{\alpha,\delta}$ and (classical) modular forms (but now of varying weights).  The naive approach of determining $\alpha$ and $\delta$ explicitly and computing $\mathcal{F}_{\alpha,\delta}$ from the original mock modular form $M^+$ runs into a number of problems.  First of all, the number $\alpha$ is a (formal) linear combination of a (probably) transcendental number plus a $p$-adic limit.  While the $p$-adic limit could be reasonably approximated, given the algebraic coefficients of the first correction $\mathcal{F}_{\alpha_1}$, one can only compute approximations to the transcendental part of $\alpha$.  One might be able to approximate $\alpha$ well enough, knowing that the resulting coefficients are rational (and somehow approximating the denominators).  To avoid this problem, we access the $p$-adic modular form more directly; we begin at an intermediary step in the proof of the existence of the $p$-adic modular form.  It turns out that although the $p$-adic modular form itself depends on the existence of $\alpha$ and $\delta$, it can be computed independently of $\alpha$ and $\delta$ (see \eqref{eqn:Fadfinal}).  We then prove the following explicit congruence for the mock modular form associated to the weight $12$ cusp form $\Delta(z):=q\prod_{n=1}^{\infty} \left(1-q^n\right)^{24}=\sum_{n\geq 1} \tau(n)q^n$, with $q:=e^{2\pi i z}$.
\begin{theorem}\label{thm:main}
Suppose that $p\geq 3$ is prime and $M$ is the harmonic weak Maass form with a simple pole at the cusp $i\infty$ (in its holomorphic part) for which $\xi_{-10}(M)=\Delta/\|\Delta\|^2$.  If $0<\ord_p\left(\tau(p)\right)=v<(k-1)/2$, then there is a unique $\alpha\in\R+\widehat{\overline{\Q_p}}$ and $\delta\in\widehat{\overline{\Q_p}}$ such that $\mathcal{F}_{\alpha,\delta}$ is a $p$-adic modular form.  

\begin{enumerate}
\item
For this unique choice of $\alpha$ and $\delta$ and for every $\ell\in\N$, there exists a weight $2+k_p p^{\ell_1+1}$ holomorphic modular form $g_{\ell}$ on $\Gamma_0(p)$ such that 
$$
\mathcal{F}_{\alpha,\delta}\Delta \equiv g_{\ell}\pmod{p^{\ell}},
$$
where $k_p=p-1$ for $p>3$, $k_3=4$, and $\ell_1:=\max\{ 11/v + \ell/v,\ell-1\}$. 
\item
For $p=3$, we have
\begin{align*}
3^7\mathcal{F}_{\alpha,\delta}\Delta&\equiv 1\pmod{3},\\
3^7\mathcal{F}_{\alpha,\delta}\Delta&\equiv E_2\pmod{3^2},\\
3^7\mathcal{F}_{\alpha,\delta}\Delta&\equiv E_2+9\Delta\pmod{3^3},
\end{align*}
where $E_2(z):=1-24\sum_{n=1}^{\infty}\sigma_1(n)q^n$ with $\sigma_r(n)=\sum_{d\mid n}d^r$.
\end{enumerate}
\end{theorem}
\begin{remark}
The congruence for $p=3$ was indicated by computational evidence in \cite{BGK}.  Note that the factor $3^7$ in front of the left-hand side corrects a typo from renormalization.
\end{remark}
By beginning in the intermediary step of the proof from \cite{BGK}, we are able to write $\mathcal{F}_{\alpha,\delta}$ as an explicit infinite ($p$-adic) linear combination of weakly holomorphic modular forms.  Only finitely many of the terms are non-zero modulo $p^{\ell}$, giving a congruence to an explicit weakly holomorphic modular form (with $p$-adic coefficients).  It is hence useful to construct weakly holomorphic modular forms.

Denote the dimension of $S_k$ by $d:=\dim_{\C}\left(S_k\right)$.  Then there is a unique weakly holomorphic modular form $f_{m,2-k}\in M_{2-k}^!$ for which
$$
f_{m,2-k}(z)=q^{-m}+O\left(q^{-d}\right).
$$
Such weakly holomorphic modular forms are explicitly constructed by Duke and Jenkins \cite{DukeJenkins08} as
\begin{equation}\label{eqn:fmexplicit}
f_{m,2-k}(z) := \left \{ \begin{array}{cl} 
E_{k'}(z) \Delta(z)^{-d-1} F_m( j(z)) & \text{ if }m > d,\\
0 & \text{ if } m\leq d,
\end{array} \right.
\end{equation}
where $k'\in \{ 0,4,6,8,10,14\}$ with $k'\equiv 2-k\pmod{12}$, $E_{k'}$ is the Eisenstein series of weight $k'$, and $F_m$ is a generalized Faber polynomial of degree $m-d-1$ chosen recursively in terms of $f_{m',2-k}$ with $m'<m$ to cancel higher powers of $q$.

The connection between $\mathcal{F}_{\alpha,\delta}$ and this basis of weakly holomorphic modular forms arises via the theory of Hecke operators.  In particular, we construct a weakly holomorphic modular form $R_p$ related to $M^+$ which is a constant multiple of $f_{p,-10}$ (see Lemma \ref{lem:RpPP}).  It is through this function $R_p$ that $\mathcal{F}_{\alpha,\delta}$ is determined in \eqref{eqn:Fadfinal}.

\section*{Acknowledgements}
The authors thank Yujie Xu for pointing out some typos in an earlier version.

\section{Preliminaries}
Here we give some preliminary definitions of the objects considered in this paper.   
\subsection{Harmonic weak Maass forms and mock modular forms}
Firstly, a weight $\kappa\in\Z$ \begin{it}harmonic weak Maass form\end{it} on $\SL_2(\Z)$ is a real analytic function on $\H$ which satisfies weight $\kappa$ modularity for $\SL_2(\Z)$, is annihilated by the weight $\kappa$ hyperbolic Laplacian 
$$
\Delta_{\kappa}:=-y^2\left(\frac{\partial^2}{\partial x^2} + \frac{\partial^2}{\partial y^2}\right) +i\kappa y \left(\frac{\partial}{\partial x}+ i\frac{\partial}{\partial y}\right),
$$
where $z=x+iy\in\H$, and grows at most linear exponentially towards $i\infty$.  There are natural subspaces $M_{\kappa}^!$ of weakly holomorphic modular forms, $M_{\kappa}$ of holomorphic modular forms, and $S_{\kappa}$ of cusp forms.  For $N\in\N$, we also denote the corresponding spaces of forms which are modular on $\Gamma_0(N)$ (elements of $\SL_2(\Z)$ whose bottom-left entry is divisible by $N$) by $M_{\kappa}^!(N)$, $M_{\kappa}(N)$, and $S_{\kappa}(N)$.  

Since a harmonic weak Maass form is fixed by $T:=\left(\begin{smallmatrix}1&1\\ 0 &1\end{smallmatrix}\right)$, it has a Fourier expansion of the type 
$$
\sum_{n\in\Z} a_n(y) q^n,
$$
where we note that $a_n(y)$ may depend on $y$.  As mentioned in the introduction, the Fourier expansions of harmonic weak Maass forms have natural decomposition into \begin{it}holomorphic parts\end{it} and a \begin{it}non-holomorphic parts\end{it}.  Specifically, we expand a harmonic weak Maass form $M$ as 
$$
M(z)=\sum_{n\gg -\infty} a_M^+(n) q^n + a_M^-(0)y^{1-\kappa} + \sum_{\substack{n\ll -\infty\\ n\neq 0}} a_M^-(n) \Gamma\left(1-\kappa; 4\pi n y\right)q^n.
$$
We then call 
$$
M^+(z):=\sum_{n\gg -\infty} a_M^+(n) q^n
$$
the holomorphic part or \begin{it}mock modular form\end{it} associated to $M$, and 
\begin{equation}\label{eqn:M-def}
M^-(z):=a_M^-(0)y^{1-\kappa} + \sum_{\substack{n\ll -\infty\\ n\neq 0}} a_M^-(n) \Gamma\left(1-\kappa; 4\pi n y\right)q^n,
\end{equation}
where $\Gamma(k;y):=\int_0^y t^{k-1}e^{-t}dt$ is the \begin{it}incomplete gamma function\end{it}.  Throughout this paper, we only consider those forms where $a_M^-(n)=0$ for $n\geq 0$.  Since $\xi_{\kappa}$ annihilates the holomorphic part, one can compute that these are precisely the harmonic weak Maass forms for which 
$$
\xi_{\kappa}(M)=\xi_{\kappa}\left(M^-\right)=\sum_{n<0}(4\pi n)^{1-\kappa} \overline{a_{M}^-(n)} q^{-n}\in S_{2-\kappa}
$$
is a cusp form.  Noting that $\Gamma\left(1-\kappa; 4\pi n y\right)q^n$ decays for $n<0$, the singularities of $M$ towards $i\infty$ are all in the holomorphic part, and we call 
$$
\sum_{n<0} a_M^+(n) q^n
$$
the \begin{it}principal part\end{it} of $M$.  For $\kappa\leq 0$, there is a unique harmonic weak Maass forms with any arbitrary principal part, which can be constructed, for example via what are known as \begin{it}Poincar\'e series\end{it}.  The $n$th Poincar\'e series $F(1,\kappa;z)$ corresponds to the (unique) harmonic weak Maass form with principal part $q^{-n}$.  

As defined in \cite{BruinierOnoRhoades}, we say that a weight $\kappa$ harmonic weak Maass form $M$ is \begin{it}good for\end{it} the cusp form $g\in S_{2-\kappa}$ if the following conditions are satisfied:
\begin{itemize}
\item[(i)] The principal part of $M$ at the cusp $\infty$ belongs to $K_g[q^{-1}]$, where $K_g$ is the field generated by the coefficients of $g$.
\item[(ii)] We have that $\xi_{2-k}(M)=\|g\|^{-2}g$.
\end{itemize}

In addition to the differential operator $\xi_{\kappa}$, the holomorphic differential operator $D^{1-\kappa}$ with $D:=\frac{1}{2\pi i } \frac{\partial}{\partial z} = q\frac{d}{dq}$ (and $1-\kappa>0$) naturally appears in the theory of harmonic weak Maass forms.  The operator $D^{1-\kappa}$ annihilates every term in \eqref{eqn:M-def} except for then $n=0$ term (giving a constant).  Since one directly obtains
$$
D^{1-\kappa}\left(M^+(z)\right) = \sum_{n\gg -\infty} n^{1-\kappa}a_M^+(n) q^n
$$
and $D^{1-\kappa}$ sends weight $\kappa$ harmonic weak Maass forms to weight $2-\kappa$ weakly holomorphic modular forms due to Bol's identity \cite{Bol}, we see a direct relation between the coefficients of $M^+$ and its image under $D^{1-\kappa}$.  In particular, working backwards from a weight $2-\kappa$ cusp form $g$ and denoting the $n$th Fourier coefficient of $g$ by $a_g(n)$, the \begin{it}(holomorphic) Eichler integral\end{it}
\begin{equation}\label{Egdefeqn}
E_g(z):=\sum_{n\geq 1} n^{\kappa-1} a_g(n) q^n
\end{equation}
is a pre-image of $g$ under $D^{1-\kappa}$ (and actually turns out to be part of a more general harmonic weak Maass form where the singularities lie in the non-holomorphic part).

\subsection{$p$-adic modular forms}
Let $p$ be a prime and fix an algebraic closure $\overline{\Q}_p$ of $\Q_p$ along with an embedding $\iota: \ \overline{\Q} \hookrightarrow \overline{\Q}_p$.  We let $\Qp$ denote the $p$-adic closure of $\overline{\Q}_p$ and normalize the $p$-adic order so that $\ord_p(p)=1$.  We do not distinguish between algebraic numbers and their images under $\iota$.  In particular, for algebraic numbers $a,b \in \overline{\Q}$ we write $a \equiv b \pmod{p^m}$ if $\ord_p(\iota(a-b)) \geq m$.  For a formal power series $H(q)=\sum_{n\in \Z} a(n) q^n \in \Qp[[q,q^{-1}]]$, we write $H \equiv 0 \pmod {p^m}$ if $\displaystyle \inf_{n\in \Z} \left(\ord_p(a(n))\right) \geq m$.

In this paper, a \begin{it}$p$-adic modular form\end{it} of level $N$ will refer to a formal power series $H(q)={\displaystyle \sum_{n\geq -t}} a(n) q^n$ with coefficients in $\Qp$ satisfying the following condition:  for every $m\in \N$ there exists a weakly holomorphic modular form ($q:=e^{2\pi i z}$ for $z\in \mathbb{H}$) $H_m(z)={\displaystyle \sum_{n\geq - t}} b_m(n) q^n\in M_{\ell_m}^!(N)$, with algebraic coefficients $b_m(n)\in \overline{\Q}$, which satisfies the congruence
\begin{equation}
H\equiv H_m\pmod{p^m}.
\end{equation}
If $\ell_m=\ell$ is constant, then we will refer to $H$ as a $p$-adic modular form of weight $\ell$, level $N$.

The main result in \cite{BGK} was the proof of the existence of $p$-adic modular forms of level $Np$ constructed from mock modular forms and Eichler integrals of level $N$.  We only consider the case $N=1$ here and briefly recall this connection.

In Theorem 1.1 of \cite{GKO}, Guerzhoy, Kent, and Ono showed that, for $M$ good for the Hecke eigenform $g\in S_k$,
$$
\FF_{a_M^+(1)}:=M^+ - a_M^+(1) E_g
$$
has coefficients in $K_g$.  Hence for every $\gamma\in \Qp$, 
\begin{equation}\label{FFadefeqn}
\FF_{\alpha}:=M^+ - \alpha E_g:= \FF_{a_M^+(1)}- \gamma E_g
\end{equation}
is an element of $\Qp[[q,q^{-1}]]$.  Here $\alpha:=a_M(1)+\gamma\in A_M$, where
$$
A_M:=a_M(1)+\Qp:=\left\{ a_M(1)+x:x\in \Qp\right\}
$$
denotes a set of formal sums.  They then investigated congruences between $\FF_{\alpha}|U_p$ and $E_g$, where $U_p$ is the usual $U$-operator defined on formal power series by 
$$
\left(\sum_{n\in\Z}a(n) q^n\right)\bigg| U_p:=\sum_{n\in\Z}a(np) q^{np}.
$$
They showed that, up to possibly one exceptional $\alpha$, the $p$-adic limit
$$
\lim_{n\to\infty} D^{k-1}\left(\FF_{\alpha}|U_p^n\right)
$$
is essentially $g$.  

In \cite{BGK}, the case of the exceptional $\alpha$ was instead investigated, and the construction was extended to consider formal power series of the type
\begin{equation}\label{FFaddefeqn}
\FF_{\alpha,\delta}:=\FF_{\alpha}-\delta\left(E_{g}-\beta E_{g|V_p}\right),
\end{equation}
where $\delta\in \Qp$.  Here $V_p$ is the usual $V$-operator defined on formal power series by 
$$
\left(\sum_{n\in\Z}a_n q^n\right)\bigg| V_p:=\sum_{n\in\Z}a_n q^{np}.
$$
Theorem 1.1 (2) of \cite{BGK} then states that there is exists unique pair $(\alpha,\delta)\in A_M\times \Qp$ (here $\alpha$ is the exceptional $\alpha$ from \cite{GKO}) such that $\mathcal{F}_{\alpha,\delta}$ is a $p$-adic modular form of weight $2-k$.  

It was noted in \cite{BGK} that, after multiplying by an appropriate power of $\Delta$, the $p$-adic modular form of weight $2-k$ would satisfy congruences with classical modular forms.  It is the goal of this paper to investigate such congruences explicitly.

\section{A Maass-Poincar\'e series}

For a harmonic weak Maass form $M$ which is good for a Hecke eigenform $g\in S_k$ and $n\in\N$, we define
$$
R_{n}(z):=M|_{2-k}T(n) - n^{1-k}a_g(n) M.
$$
The function $R_n$ is a weakly holomorphic modular form of weight $2-k$.  
We are particularly interested in the case that $n=p$ is prime and $M(z)=F(1,-10;z)$ is the first (Maass-)Poincar\'e series of weight $2-k=-10$ for $\SL_2(\Z)$. 
\begin{lemma}\label{lem:RpPP}
Let $p$ be a prime and $M(z)=F(1,-10;z)$.  Then one has 
$$
R_p(z)=p^{1-k} f_{p,-10}(z).
$$
Furthermore, the principal part of $R_p$ equals $p^{1-k} q^{-p}-p^{1-k}\tau(p)q^{-1}+O(1)$.
\end{lemma}
\begin{proof}
Actually, we will prove the second statement directly, and then the first statement will follow by uniqueness, since we know that $R_p\in M_{-10}^!$.

We look at the action of the Hecke operator on the coefficients of the holomorphic part, since we know that the non-holomorphic part cancels from the fact that $\Delta$ is a Hecke eigenform (this being the proof that $R_p\in M_{-10}^!$).  Denote the coefficients of the holomorphic part of $F(1,-10;z)$ by $a_n$ and the coefficients of $R_p$ by $b_n$.  For $n<0$ we have 
$$
a_n = \begin{cases} 1 & n=1,\\
0 &\text{otherwise}.
\end{cases}
$$
By the action of the weight $2-k$ Hecke operator coefficientwise (cf. p. 21 of \cite{OnoBook}), we have
$$
b_n = a(pn) +p^{1-k}a\left(\frac{n}{p}\right) - p^{1-k}\tau(p) a(n).
$$
with $a(s)=0$ by convention for every $s\in \Q\setminus\Z$.  Therefore whenever $n\notin\{1,p\}$ one sees immediately that $b_n=0$.  For $n=p$ one has 
$$
b_p=p^{1-k},
$$
while for $n=1$ one has 
$$
b_1=-p^{1-k}\tau(p).
$$
Therefore the first statement follows by uniqueness, since $p^{k-1} R_p - f_{p,-10}(z)=O(q^{-1})$, and is a weakly holomorphic modular form.  But then 
$$
p^{k-1} R_p - f_{p,-10}(z)=f_{1,-10}(z) = 0.
$$
\end{proof}

For the case at hand, the following lemma, which follows directly by Lemma \ref{lem:RpPP}, proves useful for computing Fourier coefficients.
\begin{lemma}
We have
$$
R_3(z)=3^{-11}\frac{E_{14}(z)}{\Delta^2(z)} \left(j(z)-768\right).
$$
\end{lemma}
\section{Connection to mock modular forms}

Here we compute the $p$-adic modular form  $\mathcal{F}_{\alpha,\delta}$ without actually computing $\alpha$ and $\delta$.  There are actually 2 additional related $p$-adic modular forms $\mathcal{F}_{\alpha}^*$ and $\mathcal{F}_{\alpha,\delta}^*$; although we only explicitly address $\mathcal{F}_{\alpha,\delta}$ in this paper, we work out the connection for the other two $p$-adic modular forms here for the interested reader.  Hence for a formal power series $\mathcal{F}$ with coefficients in $\widehat{\overline{\Q_p}}$, we define
$$
\mathcal{F}^*:=\mathcal{F}-p^{1-k}\beta'\mathcal{F}|V_p.
$$
Here $\beta$ and $\beta'$ are roots of the Hecke polynomial 
$$
x^2-\lambda_p x +\chi(p)p^{k-1}
$$ 
(i.e., $\lambda_p=\beta+\beta'$ and $\beta\beta'=p^{k-1}$) with $\ord_p(\beta)<\ord_p(\beta')$.   The following lemma is a refinement of the results in Theorem 1.1 of \cite{BGK} (only the third part is written differently than what was contained in \cite{BGK}).
\begin{lemma}\label{lem:Fadfinal}
Let $M$ be a harmonic weak Maass form which is good for a Hecke eigenform $g$ and denote its $\alpha$-correction and $(\alpha,\delta)$-corrections by $\mathcal{F}_{\alpha}$ and $\mathcal{F}_{\alpha,\delta}$.
\noindent

\noindent
\begin{enumerate}
\item
There is a unique $\alpha\in \R+\widehat{\overline{\Q_p}}$ such that $\mathcal{F}_{\alpha}^*$ is a $p$-adic modular form.  For this $\alpha$, we have 
$$
\mathcal{F}_{\alpha}^*=-\beta^{-1}p^{k-1}\sum_{\ell=0}^{\infty} \beta^{-\ell}p^{\ell(k-1)}R_p\Big|U_{p^{\ell}}.
$$
\item
There is a unique $\alpha\in \R+\widehat{\overline{\Q_p}}$ and $\delta\in \widehat{\overline{\Q_p}}$ such that $\mathcal{F}_{\alpha,\delta}^*|U_p$ is a $p$-adic modular form.  For this choice of $\alpha$ and $\delta$, we have

\item
If $\ord_p(\beta)<\ord_p(\beta')\neq k-1$, then there is a unique $\alpha\in \R+\widehat{\overline{\Q_p}}$ and $\delta\in \widehat{\overline{\Q_p}}$ such that $\mathcal{F}_{\alpha,\delta}$ is a $p$-adic modular form.  For this choice of $\alpha$ and $\delta$, we have
\begin{equation}\label{eqn:Fadfinal}
\mathcal{F}_{\alpha,\delta}=\frac{\beta}{\beta'-\beta} \sum_{n=2}^{\infty}\left(\beta'^n-\beta^n\right)R_p\Big|U_{p^{n-1}}.
\end{equation}
\end{enumerate}
\end{lemma}
\begin{remarks}
\noindent

\noindent
\begin{enumerate}
\item
To obtain $p$-adic congruences for $\mathcal{F}_{\alpha}^*$, we only need to compute congruences for 
$$
\beta^{-\ell}p^{\ell(k-1)}R_{p}|U_{p^{\ell}}
$$
with $\ell \geq 0$.  Similarly, to obtain $p$-adic congruences for $\mathcal{F}_{\alpha,\delta}$, we need to investigate $p$-adic congruences for a constant multiple of $R_{p}|U_{p^{\ell}}$.  This is what is undertaken in this paper.

\item
The series on the right-hand side of \eqref{eqn:Fadfinal} only necessarily converges (as a $p$-adic limit) in the case that $\ord_p\left(\lambda_p\right)>0$, or in other words $\ord_p(\beta')\neq k-1$.  This matches one of the conditions given in Theorem 1.1 (2) of \cite{BGK}, and we see that the rate of convergence is higher as $\ord_p(\lambda_p)$ increases.  In the case of weight $12$ investigated in this paper, it is hence interesting to investigate cases where $p\mid \tau(p)$, i.e., non-ordinary primes.  The list of such primes currently known is $p\in\{2,3,5,7,2411,7758337633\}$, with the last of this list added by Lygeros and Rozier \cite{LR}.  
\end{enumerate}
\end{remarks}
\begin{proof}
As in \cite{BGK}, we define the following operator on formal power series
$$
B_p:=-\beta p^{1-k}\left(1-\beta'p^{1-k}V_p\right)\left(1-\beta^{-1}p^{k-1}U_p\right).
$$
By (3.20) of \cite{BGK}, rewriting $B_p$ with (3.3) of \cite{BGK}, we have 
$$
M\Big|B_p= R_p(z)\in M_{2-k}^!.
$$
Recalling that $E_g$ is annihilated by $B_p$ by Lemma 3.2 of \cite{BGK}, we conclude that, for every $\alpha$,
$$
\mathcal{F}_{\alpha}\Big|B_p=R_p.
$$
\vspace{.01in}

\noindent
(1) By Proposition 3.3 (2) of \cite{BGK}, $\mathcal{F}_{\alpha}^*$ is a $p$-adic modular form if and only if the first term in (3.13) of \cite{BGK} vanishes.   However, if the first term vanishes, then we have 
$$
\mathcal{F}_{\alpha}^*=-\beta^{-1}p^{k-1}\sum_{\ell=0}^{\infty} \beta^{-\ell}p^{\ell(k-1)}R_p\Big|U_{p^{\ell}}.
$$
\vspace{.01in}

\noindent
(2)  Similarly, by (3.18) of \cite{BGK} and Proposition 3.3 (1) of \cite{BGK}, we have that $\mathcal{F}_{\alpha,\delta}^*|U_p$ is a $p$-adic modular form if and only if 
$$
\mathcal{F}_{\alpha,\delta}^*|U_p =-\sum_{\ell=1}^{\infty} \beta^{-\ell}p^{\ell(k-1)} R_p\big|U_{p^{\ell}},
$$
where we have plugged in (see the displayed formula before (3.21) in \cite{BGK})
$$
\widetilde{W}:=\mathcal{F}_{\alpha,\delta}\Big|B_pU_p = R_p\Big|U_p.
$$
\vspace{.01in}

\noindent
(3)  
Plugging in  (3.19) of \cite{BGK} and then (3.18) of \cite{BGK} and noting that $\mathcal{F}_{\alpha,\delta}$ is a $p$-adic modular form if and only if (3.16) and (3.17) of \cite{BGK} vanish, we obtain that the unique choice of $\alpha,\delta$ for which $\mathcal{F}_{\alpha,\delta}$ is a $p$-adic modular form satisfies
$$
\mathcal{F}_{\alpha,\delta}= -\sum_{\ell=0}^{\infty} \beta'^{-\ell-1}p^{(\ell+1)(k-1)}\mathcal{F}_{\alpha,\delta}^*\big|U_{p^{\ell+1}}=\sum_{\ell=1}^{\infty} \beta'^{-\ell}p^{\ell(k-1)}\sum_{j=1}^{\infty}\beta^{-j}p^{j(k-1)} R_p\Big|U_{p^{j+\ell-1}}.
$$
Denoting $n:=j+\ell$ and rewriting the sums, this simplifies as 
\begin{equation}\label{eqn:Fadrewrite}
\mathcal{F}_{\alpha,\delta}=\sum_{n=2}^{\infty} \sum_{\ell=1}^{n} \beta'^{-\ell} \beta^{\ell-n} p^{n(k-1)}R_p\Big|U_{p^{n-1}}.
\end{equation}
Using $p^{k-1}=\beta\beta'$, we may rewrite this as 
$$
\sum_{n=2}^{\infty} \sum_{\ell=1}^{n} \beta'^{n-\ell} \beta^{\ell} R_p\Big|U_{p^{n-1}}.
$$
Rewriting the geometric series, we have
$$
\sum_{\ell=1}^{n} \beta'^{n-\ell} \beta^{\ell} = \frac{\beta'^{n+1}}{\beta'-\beta}\left(1-\left(\frac{\beta}{\beta'}\right)^{n+1} - 1 + \frac{\beta}{\beta'}\right)=\frac{\beta}{\beta'-\beta}\left(\beta'^n-\beta^n\right).
$$
Thus \eqref{eqn:Fadrewrite} implies \eqref{eqn:Fadfinal}.

\end{proof}

\section{Fourier expansion of $R_p\big|U_{p^\ell}$}

We now return to the specific case that $M=F(1,-10;z)$ and consider the function $R_p\big|U_{p^{\ell}}$, where $\ell\in\N$.  This is a weakly holomorphic modular form of weight $-10$ on $\Gamma_0(p)$.  We next determine its principal part at the two cusps $0$ and $i\infty$ of $\Gamma_0(p)$.  By directly inspecting the Fourier expansion, as $z\to i\infty$ we have 
\begin{equation}\label{eqn:Rpellinfty}
R_{p}(z)\Big|U_{p^{\ell}} = \begin{cases} p^{1-k}q^{-1}+O(1)&\text{if }\ell=1,\\ O(1) &\text{if }\ell>1.\end{cases}
\end{equation}
In order to state the principal part at $0$, we require a little notation and a short lemma.  For $0\leq r<p^{\ell}$, let $g=g_r:=\gcd\left(r,p^{\ell}\right)$ and suppose that $b\in\Z$ is a solution to the equation 
$$
\frac{r}{g} b \equiv 1\pmod{\frac{p^{\ell}}{g}}.
$$
\begin{lemma}\label{lem:RpUpPP}
The principal part of the Fourier expansion of $R_{p}\big|U_{p^{\ell}}$ at $0$ is 
$$
p^{1-k-\ell}\sum_{r\pmod{p^{\ell}}}g_r^{2-k} e^{-\frac{2\pi i g_r nb}{p^{\ell}}} q^{-\frac{g_r^2}{p^{\ell-1}}}+p^{1-k-\ell}\sum_{r\pmod{p^{\ell}}} g_r^{2-k} e^{-\frac{2\pi i g_r nb}{p^{\ell}}} q^{-\frac{g_r^2}{p^{\ell}}}.
$$
\end{lemma}
\begin{remarks}
\noindent

\noindent
\begin{enumerate}
\item
Recall that the cusp width at $0$ is $p$, so we need $g_r^2/p^{\ell-2}\in\Z$.
\item
Since $g_r\leq p^{\ell}$, the order of pole is at most $p^{2\ell-(\ell-2)}=p^{\ell+2}$.  This is attained specifically for $r=0$.
\end{enumerate}
\end{remarks}
\begin{proof}
First recall that for any function $f$, $k\in\Z$, and $n\in\N$, 
$$
f(z)|U_{n} = \frac{1}{n}\sum_{r\pmod{n}} f\left(\frac{z+r}{n}\right) = n^{\frac{k}{2}-1}\sum_{r\pmod{n}} f(z)\Big|_k \left(\begin{matrix} 1 & r\\ 0 &n\end{matrix}\right),
$$
where $|_k$ is the usual weight $k$ slash operator.  Furthermore, a direct calculation yields that
$$
\left(\begin{matrix}1 & r\\ 0& p^{\ell}\end{matrix}\right)\left(\begin{matrix}0&-1\\ 1 &0\end{matrix}\right) \left(\begin{matrix}\frac{1}{g_r} & \frac{b}{p^{\ell}} \\ 0 &\frac{g_r}{p^{\ell}}\end{matrix}\right)\in \SL_2(\Z).
$$
Denoting $S:=\left(\begin{smallmatrix}0&-1\\ 1 &0\end{smallmatrix}\right)$, we hence see in particular for $f$ satisfying weight $2-k$ modularity for $\SL_2(\Z)$, that the expansion of $f|U_{p^{\ell}}$ at $0$ is 
\begin{multline*}
f\Big| U_{p^{\ell}}\Big|_{k}S = p^{\left(\frac{k}{2}-1\right)\ell}\sum_{r\pmod{p^{\ell}}} f\Bigg|_{k}\left(\begin{matrix}1 & r\\ 0& p^{\ell}\end{matrix}\right)\left(\begin{matrix}0&-1\\ 1 &0\end{matrix}\right) \left(\begin{matrix}\frac{1}{g_r} & \frac{b}{p^{\ell}} \\ 0 &\frac{g_r}{p^{\ell}}\end{matrix}\right)\Bigg|_{k}\left(\begin{matrix}\frac{1}{g_r} & \frac{b}{p^{\ell}} \\ 0 &\frac{g_r}{p^{\ell}}\end{matrix}\right)^{-1}\\
= p^{\left(\frac{k}{2}-1\right)\ell}\sum_{r\pmod{p^{\ell}}} f\Bigg|_{k}\left(\begin{matrix}\frac{1}{g_r} & \frac{b}{p^{\ell}} \\ 0 &\frac{g_r}{p^{\ell}}\end{matrix}\right)^{-1}=p^{\left(\frac{k}{2}-1\right)\ell}\sum_{r\pmod{p^{\ell}}} f\Bigg|_{k}\left(\begin{matrix}g_r & -b \\ 0 & \frac{p^{\ell}}{g_r}\end{matrix}\right).
\end{multline*}
To compute the expansion at $0$, we then simplify
\begin{equation}\label{eqn:fcusp0}
f(z)\Big| U_{p^{\ell}}\Big|_{k}S=\frac{1}{p^{\ell}} \sum_{r\pmod{p^{\ell}}} g_r^{k}f\left(\frac{g_r^2}{p^{\ell}}z -\frac{g_r}{p^{\ell}}b\right).
\end{equation}
We now plug in the principal part $\sum_{-n_0<n<0}a_n q^n$ of $f$ to rewrite the right-hand side of \eqref{eqn:fcusp0} as
$$
\frac{1}{p^{\ell}}\sum_{-n_0<n<0}a_n\sum_{r\pmod{p^{\ell}}} g_r^{k}e^{-\frac{2\pi i g_r nb}{p^{\ell}}} q^{\frac{g_r^2 n}{p^{\ell}}}+O(1).
$$
Specifically, for $f=R_p$ (and $k\to 2-k$), Lemma \ref{lem:RpPP} implies that the principal part of $R_p|U_p^{\ell}$ is 
$$
p^{1-k-\ell}\sum_{r\pmod{p^{\ell}}} g_r^{2-k} e^{-\frac{2\pi i g_r nb}{p^{\ell}}} q^{-\frac{g_r^2}{p^{\ell-1}}}+p^{1-k-\ell}\sum_{r\pmod{p^{\ell}}} g_r^{2-k} e^{-\frac{2\pi i g_r nb}{p^{\ell}}} q^{-\frac{g_r^2}{p^{\ell}}}.
$$
\end{proof}

\section{Congruences}
For $k>2$ even, we first consider the properties of 
$$
E_{k,p}(z):=E_k(z)-p^{k}E_{k}(pz)
$$
and make use of the trick of Serre \cite{Serre} that 
\begin{equation}\label{eqn:Serrecong}
E_{(p-1)k}\equiv 1\pmod{p^{1+\ord_p(2k)}}.
\end{equation}
\begin{lemma}\label{lem:Ekp}
The function $E_{k,p}$ has integral coefficients for $r\in\N$ with $r=p^md$ we have
$$
E_{k,p}^r\equiv E_k^r\pmod{p^{k+m}}.
$$
  Furthermore, $E_{k,p}$ vanishes at the cusp zero.
\end{lemma}
\begin{remark}
Note that in particular, for $p=3$, \eqref{eqn:Serrecong} implies that
$$
E_{6,3}\equiv E_6\equiv 1\pmod{3^2}.
$$
More generally, for $r\in\N_0$ and $j\in\N$ we have
\begin{equation}\label{eqn:E63cong}
E_{6,3}^{3^rj}\equiv E_6^{3^rj}\equiv 1\pmod{3^{r+2}}.
\end{equation}
\end{remark}
\begin{proof}
Since $E_k$ has coefficients in $\Z$, we see immediately that $E_{k,p}\equiv E_k\pmod{p^k}$.  The congruence for higher powers now follows immediately.  

The Fourier expansion of $E_{k,p}$ at the cusp zero is given by (using the modularity of $E_k$ on $\SL_2(\Z)$)
$$
E_{k,p}\Big|_k S(\tau) = E_k\Big|_k S(z) - p^kz^{-k} E_k\left(-\frac{p}{z}\right) = E_k(z) -E_k\left(\frac{z}{p}\right).
$$
We see immediately that the constant term vanishes.  
\end{proof}

In order to obtain interesting congruences, we multiply $\mathcal{F}_{\alpha,\delta}$ by $\Delta$ to kill the pole at $i\infty$ and then multiply by powers of $E_{k,p}$ to kill off the pole at $0$.  Doing so, we obtain congruences with classical modular forms.  
\begin{lemma}\label{lem:Ekppow}
For a holomorphic modular form $g\in M_{2+kr}(p)$, we have 
$$
 R_p|U_{p^{\ell}}\Delta E_{k,p}^r\equiv g\pmod{p^{\ell}}
$$
if and only if the congruence holds for the first
\begin{equation}\label{eqn:genbound}
\frac{p+1}{6} + \frac{k\left(p^{\ell+2}+p^{\ell+1}\right)}{12}
\end{equation}
coefficients.
\end{lemma}
\begin{proof}
By Lemma \ref{lem:Ekp}, $E_{k,p}$ has integral coefficients (in its Fourier expansion over $i\infty$) and is invertible as a formal power series modulo $p^{\ell}$ because the constant term is $1+p^4$.  Hence for any $r\in\Z$
$$
f\equiv g\pmod{p^{\ell}}\Leftrightarrow fE_{k,p}^{r}\equiv gE_{k,p}^{r}\pmod{p^{\ell}}. 
$$
In particular, since for $\ell\geq 1$, the function $R_p|U_{p^{\ell}}$ has a pole of order at most one at $z$ and at most $p^{\ell+2}$ at $0$, the function 
$$
 R_p|U_{p^{\ell}}\Delta E_{k,p}^{p^{\ell+1}}
$$
is a holomorphic modular form of weight $2+kp^{\ell+1}$.  If $g$ is a holomorphic modular form of weight $2+kr$, then 
$$
R_p|U_{p^{\ell}}\Delta E_{k,p}^r\equiv g\pmod{p^{\ell}}
$$
if and only if 
\begin{equation}\label{eqn:holcong}
R_p|U_{p^{\ell}}\Delta E_{k,p}^{p^{\ell+1}}\equiv gE_{k,p}^{p^{\ell+1}-r}\pmod{p^{\ell}}.
\end{equation}
Since the congruence \eqref{eqn:holcong} is between holomorphic modular forms of weight $2+kp^{\ell+1}$ on $\Gamma_0(p)$, Sturm's bound \cite{Sturm} implies the congruence \eqref{eqn:holcong} if and only if it holds for the first 
$$
\left(\frac{1}{6} + \frac{kp^{\ell+1}}{12}\right)\left[\SL_2(\Z):\Gamma_0(p)\right]
$$
coefficients.  Moreover, the index of the subgroup is (cf. Proposition 1.7 of \cite{OnoBook}) 
$$
\left[\SL_2(\Z):\Gamma_0(p)\right]=p+1.
$$
From this we directly obtain \eqref{eqn:genbound}.
\end{proof}

We are finally ready to prove Theorem \ref{thm:main}.
\begin{proof}[Proof of Theorem \ref{thm:main}]
(1)  By Lemma \ref{lem:Fadfinal} (and in particular \eqref{eqn:Fadfinal}), for the unique choice of $\alpha$ and $\delta$ giving a $p$-adic modular form (its existence is proven in Theorem 1.1 (2) of \cite{BGK}), we have 
$$
\mathcal{F}_{\alpha,\delta}=\frac{\beta}{\beta'-\beta}\sum_{n=2}^{\infty}\left(\beta'^n-\beta^n\right)R_p\Big|U_{p^{n-1}}.
$$
By Lemma \ref{lem:RpPP}, $p^{k-1}R_p$ has integral coefficients.  If $\ord_p(\lambda_p)=\ord_p(\beta)=v$, then we conclude that 
$$
\mathcal{F}_{\alpha,\delta}\Delta\equiv \frac{\beta}{\beta'-\beta} \sum_{n=2}^{r}\left(\beta'^n-\beta^n\right)R_p\Big|U_{p^{n-1}}\pmod {p^{vr+1-k}}.
$$
We now choose $r_{\ell}$ such that $v r_{\ell}+1-k\geq \ell$ to obtain a congruence to a weakly holomorphic modular form modulo $p^{\ell}$. 

By Lemma \ref{lem:RpUpPP} and Lemma \ref{lem:Ekp}, $R_p\Big|U_{p^{n-1}}\Delta E_{k_p,p}^{p^{n}}$ is a holomorphic modular form.  Moreover, Lemma \ref{lem:Ekp} and \eqref{eqn:Serrecong} imply that
\begin{equation}\label{eqn:FEkpcong}
\mathcal{F}_{\alpha,\delta}\Delta E_{k_p,p}^{p^{\ell_1}}\equiv \mathcal{F}_{\alpha,\delta}\Delta E_{k_p}^{p^{\ell_1}}\equiv \mathcal{F}_{\alpha,\delta}\Delta\pmod{p^{1+\ell_1}}.
\end{equation}
Therefore 
$$
\mathcal{F}_{\alpha,\delta}\Delta \equiv \frac{\beta}{\beta'-\beta}\sum_{n=2}^{r_{\ell}}\left(\beta'^n-\beta^n\right)R_p\Big|U_{p^{n-1}}\Delta E_{k_p,p}^{p^{\ell_1}} \pmod {p^{1+\ell_1}}.
$$
To obtain a congruence to a holomorphic modular form modulo $p^{\ell}$, we need $\ell_1\geq r_{\ell}$ (so that each summand is a holomorphic modular form) and $1+\ell_1\geq \ell$.  It thus remains to determine $r_{\ell}$.  We have restricted ourself to the case $k=12$, so this is 
$$
r_{\ell}\geq \frac{11+\ell}{v}.
$$
Hence we choose 
$$
\ell_1=\max\left\{\frac{11+\ell}{v}, \ell-1\right\}.
$$
The weight of the form is $2+k_p p^{\ell_1}$.

(2)  In the specific case $p=3$, we have $v=2$ and \eqref{eqn:FEkpcong}, and hence $r_{\ell}\geq \frac{11+\ell}{2}$, so that 
$$
\ell_1=\max\left\{ \frac{11+\ell}{2},\ell-2\right\}.
$$

We would like to show that 
$$
3^7\mathcal{F}_{\alpha,\delta}\Delta \equiv E_2+\Delta\pmod{3^3}
$$
from which the other congruences are immediately implied.  Note first that $E_2$ is not a modular form (it is a weight 2 mock modular form).  It is well-known that 
$$
\widehat{E}_2(z):=E_2(z)-\frac{3}{\pi y}
$$
is a weight $2$ harmonic weak Maass form, where $y=\im(z)$.  We see immediately that 
$$
\widetilde{E}_{2,p}:=E_2(z)-pE_2|V_p(z) = E_2(z)- \frac{3}{\pi y} -pE_2|V_p(z) +p\left(\frac{3}{\pi p y}\right) = \widehat{E}_2(z)-\widehat{E}_2(z)|V_p
$$
is a holomorphic modular form (because the left-hand side is holomorphic and the right-hand side satisfies weight $2$ modularity).  Hence
$$
\widetilde{E}_{2,3}^{10}+3E_{10}^2|V_3
$$
is a weight $20$ modular form and we claim that 
\begin{align}
\label{eqn:E2cong}        E_2&\equiv \widetilde{E}_{2,3}^{10}+3E_{10}^2|V_3\pmod{3^3},\\
\label{eqn:Deltacong} 9\Delta&\equiv 9\Delta \widetilde{E}_{2,3}^4\pmod{3^3},
\end{align}
so that 
\begin{equation}\label{eqn:E2Deltacong}
E_2+9\Delta\equiv \widetilde{E}_{2,3}^{10}+3E_{10}^2|V_3+9\Delta \widetilde{E}_{2,3}^4\pmod{3^3}.
\end{equation}
The congruence \eqref{eqn:Deltacong} follows immediately from $E_2\equiv 1\pmod{3}$.  Writing $E_{10}=E_4E_6$ and noting that $E_6\equiv 1\pmod{3^2}$ by \eqref{eqn:Serrecong}, we have 
$$
\widetilde{E}_{2,3}^{10}+3E_{10}^2|V_3\equiv \widetilde{E}_{2,3}^{10}+3E_{4}^2|V_3\pmod{3^3}.
$$
By the binomial theorem, one easily sees that 
\begin{equation}\label{eqn:E23cong}
\widetilde{E}_{2,3}^9\equiv 1\pmod{3^3}.
\end{equation}
It hence remains to show that 
$$
E_4^2\equiv E_2\pmod{3^2}.
$$
To show this, we simply note that by \eqref{eqn:Serrecong} and $E_2\equiv \widetilde{E}_{2,3}\equiv 1\pmod{3}$ we have
$$
E_2=\widetilde{E}_{2,3}+3E_2|V_3\equiv \widetilde{E}_{2,3}+3\equiv 4\widetilde{E}_{2,3}\equiv 4\widetilde{E}_{2,3}E_6\pmod{3^2}.
$$
By Sturm's bound \cite{Sturm}, the congruence $E_4^2\equiv 4 \widetilde{E}_{2,3}E_6\pmod{3^2}$ is verified by comparing $3$ coefficients.  We have hence shown \eqref{eqn:E2cong}, and we therefore would like to show 
$$
3^7\mathcal{F}_{\alpha,\delta}\Delta\equiv  \widetilde{E}_{2,3}^{10}+3E_{10}^2|V_3+9\Delta \widetilde{E}_{2,3}^4\pmod{3^3}.
$$
Since (see \eqref{eqn:E63cong} for the congruence)
$$
E_{6,3}^3\equiv 1\pmod{3^3}
$$
is a weight $18$ modular form on $\Gamma_0(3)$, this is equivalent to showing that 
\begin{equation}\label{eqn:toshow}
3^7\mathcal{F}_{\alpha,\delta}\Delta E_{6,3}^3 \equiv  \widetilde{E}_{2,3}^{10}+3E_{10}^2|V_3+9\Delta \widetilde{E}_{2,3}^4\pmod{3^3}.
\end{equation}
Since $3^{11}R_{3}$ has integral coefficients by Lemma \ref{lem:RpPP}, $\ord_3(\beta)=2$ and $\ord_3(\beta')=9$, we furthermore obtain by \eqref{eqn:Fadfinal} that
$$
3^7\mathcal{F}_{\alpha,\delta}\equiv -3^{-4}\frac{\beta}{\beta'-\beta} \sum_{n=2}^{4}\beta^n\cdot  3^{11}R_3\Big|U_{3^{n-1}}\pmod{3^3}.
$$
Hence \eqref{eqn:toshow} is equivalent to 
\begin{equation}\label{eqn:weaktoshow}
-3^{-4}\Delta E_{6,3}^3\frac{\beta}{\beta'-\beta} \sum_{n=2}^{4}\beta^n \left( 3^{11}R_3\Big|U_{3^{n-1}}\right) \equiv \widetilde{E}_{2,3}^{10}+3E_{10}^2|V_3+9\Delta \widetilde{E}_{2,3}^4\pmod{3^3}.
\end{equation}
In \eqref{eqn:weaktoshow}, we finally have weight $20$ weakly holomorphic modular forms on $\Gamma_0(3)$ on both sides of the congruence.  The right-hand side is furthermore a holomorphic modular form and the left-hand side has trivial principal part at $i\infty$ by \eqref{eqn:Rpellinfty}.  

By Lemma \ref{lem:RpUpPP}, the highest power of $q^{-1/3}$ occurring in the principal part of $R_3\Big|U_{3^{n-1}}$ is at most $3^{n+1}$.  The highest power of $n$ occuring is $n=4$, giving a pole of order $3^5$.  Since $\Delta$ vanishes to order $3$ (in powers of $q^{-1/3}$) at $0$ and $E_{6,3}$ vanishes to order $1$, we conclude that the highest power of $q$ in the principal part at $0$ is 
$$
q^{-237/3}=q^{-79}.
$$
We now multiply both sides of \eqref{eqn:weaktoshow} by $E_{4,3}^{237}$.  This yields a congruence between modular forms of weight $968$ on $\Gamma_0(3)$, requiring the comparison of $323$ coefficients to check by Sturm's bound \cite{Sturm}.  Note that to do so we must calculate $8721$ coefficients of $R_3$, since we apply $U_{3^3}$ to this.
 
\end{proof}

\end{document}